\documentclass[11pt]{amsart}
\usepackage[all]{xy}
\usepackage{amsmath}
\usepackage{amsfonts}
\usepackage{amssymb}
\usepackage{amscd}
\usepackage{amsthm}
\usepackage{latexsym}
\usepackage{amsbsy}

\def\0D{\Delta^{(0)}}
\def\1D{\Delta^{(1)}}

\newcommand{\Kc}{\mathcal{K}}
\newcommand{\Hc}{\mathcal{H}}

\newtheorem{theorem}{Theorem}[section]

\newtheorem{proposition}[theorem]{Proposition}
\newtheorem{lemma}[theorem]{Lemma}
\newtheorem{corollary}[theorem]{Corollary}

\newtheorem{example}[theorem]{Example}

\def\build#1_#2^#3{\mathrel{
\mathop{\kern 0pt#1}\limits_{#2}^{#3}}}
\newcommand{\ns}[1]{~\hspace{-4pt}_{_{{(#1)}}}}
\newcommand{\ps}[1]{~\hspace{-4pt}^{^{(#1)}}}

\numberwithin{equation}{section}

 \newcommand{\ie}{{\it i.e.\/}\ }

\def\d{\delta}

\def\ve{\varepsilon}

\def\ot{\otimes}
\def\part{\partial}

\def\text{\hbox}

\def\ot{\otimes}

\def\Hom{\mathop{\rm Hom}\nolimits}

\def\Id{\mathop{\rm Id}\nolimits}

\def\build#1_#2^#3{\mathrel{
\mathop{\kern 0pt#1}\limits_{#2}^{#3}}}

\numberwithin{equation}{section}
\newcommand{\comment}[1]{\relax}

\begin{document}
\title{On Cyclic Cohomology of  $\times$-Hopf algebras}
\author {Mohammad Hassanzadeh}
\curraddr{University of Windsor, Department of Mathematics and Statistics, Lambton Tower, Ontario, Canada.
}
\email{mhassan@uwindsor.ca}
\subjclass[2010]{19D55,  16T05, 11M55}
 \keywords{Cyclic cohomology, Hopf algebras, Noncommutative geometry}

\maketitle
\begin{abstract}
In this paper we study the cyclic cohomology of  certain $\times$-Hopf algebras: universal enveloping algebras, quantum algebraic tori, the Connes-Moscovici $\times$-Hopf algebroids and the Kadison bialgebroids. Introducing  their stable anti Yetter-Drinfeld modules and cocyclic modules, we compute their cyclic cohomology. Furthermore, we provide a pairing for the cyclic cohomology of $\times$-Hopf algebras which generalizes the Connes-Moscovici characteristic map to $\times$-Hopf algebras. This enables us to transfer the $\times$-Hopf algebra cyclic cocycles to algebra cyclic cocycles.

\end{abstract}

\section{ Introduction}
 A notion of $\times$-Hopf algebra which generalizes the one for Hopf algebra was introduced by  Schauenburg in \cite{Sch98}.
It extends the   notion of Hopf algebroid  of  B\"ohm-Szlach\'anyi   \cite{b1}  and many nice examples of J. H. Lu in \cite{Lu}. The notion of Hopf algebroid (in the  commutative case)   first appeared  in Haynes Miller's PhD thesis \cite{miller} and has roots in the works of Frank Adams around 1970.
 One notes that a B\"ohm-Szlach\'anyi Hopf algebroid is not necessarily a Lu Hopf algebroid and vice versa. Some nice quantum groupoids, such as weak Hopf algebras with invertible antipodes and also Khalkhali-Rangipour's para-Hopf algebroids \cite{kr3},  are examples of B\"ohm-Szlach\'anyi Hopf algebroids. Another interesting example is the Connes-Moscovici Hopf algebroid \cite{CM01} which was originally understood as a Lu Hopf algebroid and also satisfies the B\"ohm- Szlach\'anyi axioms. It is known that any B\"ohm-Szlach\'anyi's Hopf algebroid (with invertible antipode) is also a $\times$-Hopf algebra.

\medskip
Cyclic cohomology of Hopf algebras was introduced by Connes-Moscovici  in their ground breaking
work on local index theory \cite{CM98}. Their work  was followed by significant calculations of Hopf cyclic cohomology for quantum groups by  Khalkhali-Rangipour in \cite{kr1},  Kustermans-Rognes-Tuset in \cite{krt} and also  Hadfield-Kr\"ahmer in \cite{hk} and \cite{hk2}.
 In the extended version, the cyclic cohomology of Lu's Hopf algebroid (which is defined in  \cite{CM01} and \cite{kow}) and the cyclic cohomology of
Khalkhali-Rangipour's para- Hopf algebroid [KR3] are defined with
trivial coefficients. 
Cyclic cohomology theory  with  generalized  coefficients   for extended versions of Hopf algebras was first  defined in \cite{bs2} for $\times$-Hopf algebras and later was generalized in \cite{HR1}, \cite{HR2} and \cite{kk}. The authors of \cite{bs3} developed a categorial approach to find cyclic objects  for $\times$-Hopf algebras. Later on, the authors  of \cite{kp}  extended the formalism of Hopf cyclic cohomology to the context of braided categories.

\medskip
The Connes-Moscovici characteristic map  introduced in \cite{CM98}  has many  applications in index theory,  number theory and Hopf cyclic cohomology which   was introduced in \cite{CM01} and \cite{CM00} ( see also  \cite{kay00}).  This characteristic map was  discovered while computing the index of a transversally elliptic operator on a foliation.

The key idea here was that although computing the cyclic cohomology of algebras  is a difficult task, one can compute the Hopf cyclic cohomology of the related symmetry  and then use the characteristic map to transfer the cocycles and therefore information from the Hopf algebra of the symmetry to the algebra in question. Later different extensions of the Connes-Moscovici characteristic  map for Hopf algebras were introduced in \cite{kr3}, \cite{cr}, \cite{ns} and \cite{kay}. Finally,   Kaygun  in \cite{kay04}  proved that all of these different setups produce isomorphic characteristic maps.

\medskip

The current paper is organized as follows. In Section $2$ we recall the basics of $\times$-Hopf algebras, their modules and comodules and especially stable anti-Yetter-Drinfeld (SAYD) modules. Also we introduce two cocyclic modules for the cyclic cohomology of $\times$-Hopf algebras. Furthermore we provide a pairing between the cyclic cohomology of module algebras and module corings under the symmetry of a $\times$-Hopf algebra. Therefore  we obtain a generalization of the Connes-Moscovici characteristic map   for an extended version of Hopf algebras.
In Section $3$ we apply the theoretical results and notions of Section $2$ to four major examples of  $\times$-Hopf algebras: enveloping algebras, quantum algebraic tori, the Connes-Moscovici Hopf algebroids  and the Kadisson bialgebroids.

\bigskip

\textbf{Acknowledgments}: The author would like to thank the  Institut des Hautes \'Etudes Scientifiques, IHES, for its hospitality and financial support during his stay when the majority of  this work was accomplished. Also the author appreciates the Mathematics  and Statistics department of University of Windsor, Canada, where  this work was partially carried out. Finally the author would like to thank the referee for his/her constructive suggestions.
\tableofcontents

\section{Cyclic cohomology of $\times$-Hopf algebras with coefficients}
After recalling the basics of $\times$-Hopf algebras, we introduce two cocyclic modules for computing cyclic cohomology of $\times$-Hopf algebras. At the end of the section, we define a pairing on the cyclic cohomology of $\times$- Hopf algebras.

\subsection{$\times$-Hopf algebras }
In this subsection, we recall the definition and basic properties of  $\times$-Hopf algebras.

Let $R$ be an algebra over the field of complex numbers $\mathbb{C}$. A left bialgebroid $\mathcal{K}$ over $R$  consists of the data $(\Kc, \mathfrak{s}, \mathfrak{t})$. Here $\Kc$ is a $\mathbb{C}$-algebra, $\mathfrak{s}:R\longrightarrow \Kc$ and $\mathfrak{t}:R^{op}\longrightarrow \Kc$ are $\mathbb{C}$-algebra maps such that their ranges commute with one another. In terms of $\mathfrak{s}$ and $\mathfrak{t}$, $\Kc$ can be equipped with a $R$-bimodule structure as follows;
\begin{equation*}
r_{1}\cdot k \cdot r_{2}=\mathfrak{ s}(r_{1})\mathfrak{t}(r_{2})k,
 \end{equation*}
 for all $r_1, r_2\in R$ and $k\in \Kc$. Similarly, $\Kc\ot_R\Kc$ is endowed with a natural $R$-bimodule structure. Also we assume that there are $R$-bimodule maps $\Delta: \Kc\longrightarrow \Kc\ot_R \Kc$ called coproduct and $\ve: \Kc\longrightarrow R$ called counit via which $\Kc$ is a $R$-coring \cite{bwbook}. For the coproduct we introduce the Sweedler summation notation $\Delta(k)=k\ps{1}\otimes_R k\ps{2},$ where implicit summation understood. The data $(\Kc, \mathfrak{s},\mathfrak{t}, \Delta, \ve)$ is called a left $R$-bialgebroid if  the algebra and the coring structures have  the following compatibility axioms for all $k,k'\in\Kc$ and $r\in R$;
\begin{enumerate}
\item[i)] $k\ps{1} \mathfrak{t}(r)\otimes_{R} k\ps{2} = k\ps{1} \otimes_{R} k\ps{2} \mathfrak{s}(r),$
\item[ii)] $\Delta(1_{\Kc})=1_{\Kc} \ot_R 1_{\Kc}$, and $\Delta(kk')=k\ps{1} k'\ps{1}\otimes_R k\ps{2}k'\ps{2},$
\item[iii)] $\varepsilon(1_{\Kc})=1_{R}$ and $\varepsilon(kk')=\varepsilon(k\mathfrak{s}(\varepsilon(k')))$.
\end{enumerate}

 A left $R$-bialgebroid $(\Kc,\mathfrak{s},\mathfrak{t},\Delta,\varepsilon)$, is said to be a left $\times_{R}$-Hopf algebra if the following map
\begin{equation}\label{aa}
\nu: \mathcal{\Kc}\otimes_{R^{op}}\mathcal{\Kc}\longrightarrow \mathcal{\Kc}\otimes_{R}\mathcal{\Kc}, \qquad  k\otimes_{R^{op}}k'\longmapsto k\ps{1}\otimes_{R}k\ps{2}k'
\end{equation}

is bijective. In the domain of the map \eqref{aa}, $R^{op}$-module structures are given by right and left multiplication by $\mathfrak{t}(r)$ for $r\in R$. In the codomain of the map \eqref{aa}, $R$-module structures are given by right multiplication by $\mathfrak{s}$ and $\mathfrak{t}$.
 The maps $\nu$ and $\nu^{-1}$ are both right $\mathcal{\Kc}$-linear. The image of $\nu^{-1}$ is denoted by
 \begin{equation*}
 \nu^{-1}(h\otimes_{R^{op}}1_{\mathcal{H}})=h^{-}\otimes_{R}h^{+}.
  \end{equation*}
  The notation of a $\times$-Hopf algebra extends that of a Hopf algebra. In fact if $\mathcal{\Kc}$ is a bialgebra,  then bijectivity of the map $\nu$ is equivalent to the fact that $\mathcal{\Kc}$ is a Hopf algebra. In this case, the inverse of the map $\nu$ is defined by
  \begin{equation*}
\nu^{-1}(h\otimes 1)= h^{-}\otimes_{R^{op}} h^{+}= h\ps{1}\otimes_{R^{op}}S(h\ps{2}).
 \end{equation*}
  We note that in the bialgebroid structure we have  equal source and target maps $\mathfrak{s}=\mathfrak{t}: \mathbb{C}\longrightarrow \Kc$ given by $c\longmapsto c 1_{\Kc}$. The following lemma \cite{HR1}, which implies some properties of the map $\nu$ will be used in the next subsection.

  \begin{lemma}\label{lemma-proof}
    For any left $\times_R$-Hopf algebra $\Kc$   the following identities hold for all $k\in \Kc$.
    \begin{enumerate}
\item[\rm i)] $k^{-}\ps{1}\ot_R k^-\ps{2}k^+= k\ot_R 1$.
\item[\rm ii)] $k^-\ps{1}\ot_R k^-\ps{2}\ot_{R^{op}} k^+=k\ps{1} \ot_R {k\ps{2}}^-\ot_{R^{op}}{k\ps{2}}^+$
    \end{enumerate}
  \end{lemma}

  Here we briefly recall the definitions of  modules, comodules and stable anti-Yetter-Drienfeld (SAYD) modules for a left $\times$-Hopf algebra.
A right module over a left $\times_{R}$-Hopf algebra $\Kc$ is a right $\mathcal{K}$-module $M$. A right $\mathcal{K}$-module $M$ can be equipped with a $R$-bimodule structure as follows:
$$r\cdot m=s(r)\cdot m, \quad \text{ and} \quad m \cdot r=t(r)\cdot m $$

 A left comodule of a left $\times_{R}$-Hopf algebra $\Kc$ is defined to be a left comodule of the underlying $R$-coring $(\mathcal{K},\Delta,\varepsilon)$, that is, a left $R$-module $M$, together with a left $R$-module map $M\longrightarrow \mathcal{K}\otimes_{R}M,$ $m\longmapsto m\ns{-1}\otimes_{R}m\ns{-1},$ satisfying coassociativity and counitality axioms.
 One notes that a left $\mathcal{K}$-comodule $M$ can be equipped with a $R$-bimodule structure by introducing a right $R$-action as follows,
\[m\cdot r:=\varepsilon(m\ns{-1}s(r))\cdot m\ns{0} \label{can},\] for $r\in R$ and $m\in M.$ With respect to the resulting bimodule structure, $\mathcal{K}$-comodule maps are $R$-bimodule maps. In the special case, the left $\mathcal{\Kc}$-coaction on $M$ is an $R$-bimodule map in the sense that for all $r,r'\in R$ and $m\in M$, we have;
\begin{equation}
(r\cdot m \cdot r')\ns{-1}\otimes_{R} (r\cdot m \cdot r')\ns{0}= s(r)m\ns{-1}s(r')\otimes_{R} m\ns{0}.
\end{equation}
 Furthermore,  for all $m\in M$ and $r\in R$ we have;
 \begin{equation}
m\ns{-}\otimes_{R} m\ns{0}\cdot r= m\ns{-1}t(r)\otimes_{R}m\ns{0}.
\end{equation}
 Let $M$ be a right $\mathcal{\Kc}$-module and a left $\mathcal{K}$-comodule.  We say $M$ is an anti Yetter-Drinfeld, AYD, module provided that the following two conditions hold.\\

$i)$ The $R$-bimodule structures on $M$, underlying its module and comodule structures, coincide. That is, for $m\in M$ and $r\in R$,
$$m \cdot r= m \mathfrak{s}(r), \qquad \text{and} \qquad r \cdot m= m  \mathfrak{t}(r),$$ where $r\cdot m$ denote the left $R$-action on the left $\mathcal{K}$-comodule $M$ and $r\cdot m$ is the canonical right action.\\

$ii)$ For $k\in\mathcal{K}$ and $m\in M$ we have;
\begin{equation}\label{SAYD-condition}
(m k)\ns{-1}\otimes(m  k)\ns{0}= {k\ps{2}}^{+} m\ns{-1}k\ps{1}\otimes_{R} m\ns{0} {k\ps{2}}^{-}.
\end{equation}
 This condition is known as the anti Yetter-Drinfeld (AYD) condition. The anti Yetter-Drinfeld module $M$ is said to be stable if in addition for any $m\in M$ we have $m\ns{0}m\ns{-1}=m.$\\

\begin{example}\label{sayd-eg1}{\rm \emph{\textbf{SAYD module structure on the ground algebra $R$ }}:
A map  $\d$ is called a  right character \cite[Lemma 2.5]{b1}, for a  left $\times_R$-Hopf algebra $\mathcal{K}$ if it satisfies the following conditions:
\begin{align}\label{char-1}
& \delta(k \mathfrak{s}(r))= \delta(k)r, \quad \text{for} \quad k\in \mathcal{K} \quad \text{and} \quad r\in R,\\\label{char-2}
& \delta(k_1k_2)=\delta(\mathfrak{s}(\delta(k_1))k_2), \quad \text{ for} \quad k_1,k_2\in \mathcal{K},\\\label{char-3}
& \delta(1_{\mathcal{K}})=1_R.
\end{align}
As an example, for any right $\times_R$-Hopf algebra, the counit $\varepsilon$ is a right character. Now we recall the SAYD structure of $R$ from \cite[Example 2.18]{bs2} and \cite[Example 2.5]{HR1}. Let $\sigma \in \mathcal{K}$  be a group-like element and the map $\d:\mathcal{K}\longrightarrow R$ be a right character.
The following action and coaction,
\begin{align}\label{RSAYD}
r \triangleleft k= \d(\mathfrak{s}(r)k), \quad \text{and} \quad r\longmapsto \mathfrak{s}(r)\sigma \ot 1
\end{align}

define a right $\mathcal{K}$-module and left $\mathcal{K}$-comodule structure  on $R$, respectively. These action and coaction amount  to a right-left anti Yetter-Drinfeld module on $R$ if and only if, for all $r\in R$ and $k\in \mathcal{K}$ we have;
\begin{align}
\mathfrak{s}(\d(k))\sigma=\mathfrak{t}(\d({k\ps{2}}^-))k^{(2)+}\sigma k\ps{1}, \quad \text{and} \quad \varepsilon(\sigma \mathfrak{s}(r))=\d(\mathfrak{s}(r)).
\end{align}
The anti Yetter-Drinfeld module $R$ is stable if in addition we have $\d(\mathfrak{s}(r)\sigma)=r$, for all $r\in R$. We denote this SAYD module over $\Kc$ by $^\sigma R_{\delta}$.
}
\end{example}
\subsection{Cocyclic modules for $\times$-Hopf algebras}

In this subsection we introduce two cocyclic modules for $\times$-Hopf algebras under the symmetry of  module algebras and module corings.

 For any left $\times_R$-Hopf algebra  $\mathcal{K}$ a left $\mathcal{K}$-module coring  $C$ is a $R$-coring and a left $\mathcal{K}$-module with one and the same underlying $R$-bimodule structure, such that counit $\varepsilon$ and comultiplication $\Delta$  are both left $\mathcal{K}$-linear. We consider the left $\mathcal{K}$-module structure of $R$ by $k\vartriangleright r:= \varepsilon(ks(r))$ and left $\mathcal{K}$-module structure of $C\otimes_{R}C$ is defined by the diagonal action. This means;
 \begin{align}
&\varepsilon(k\vartriangleright c)= k\vartriangleright \varepsilon(c) = \varepsilon(ks(\varepsilon(c))),\\
&\Delta(k\vartriangleright c)= k\ps{1} \vartriangleright c\ps{1} \otimes_{R}k\ps{2} \vartriangleright c\ps{2}.
\end{align}
 For any  left $\mathcal{K}$-module coring $C$ and  a right-left SAYD module $M$ over $\mathcal{K}$, one  defines a cocyclic module as follows. We set
$$_{\Kc}C^{n}(C,M)= M\otimes_{\mathcal{K}}C^{\otimes_{R}(n+1)}.$$ Here  $\Kc$ acts on $C^{\ot_R(n+1)}$ by the diagonal action.  We abbreviate $\widetilde{c}=c_{0}\otimes_{R}\cdots \otimes_{R}c_{n}$ and define the following cofaces, codegeneracies and cocyclic maps.
\begin{eqnarray}\label{cocyclic-module-1}
&&d_{i}(m\otimes_{\mathcal{K}}\widetilde{c})= m\otimes_{\mathcal{K}}c_0\otimes_{R}\cdots \otimes_{R} \Delta(c_{i}) \otimes_{R}\cdots \otimes_{R} c_{n}, \nonumber \\
&&d_{n+1}(m\otimes_{\mathcal{K}}\widetilde{c})=m\ns{0}\otimes_{\mathcal{K}}c_{0}\ps{2}\otimes_{R}c_1\otimes_{R}\cdots\otimes_{R}c_{n}\otimes_{R}m\ns{-1}c_{0}\ps{1},\nonumber \\
&&s_{i}(m\otimes_{\mathcal{K}}\widetilde{c})=m\otimes_{\mathcal{K}}c_{0}\otimes_{R}\cdots\otimes_{R}\varepsilon(c_{i})\otimes_{R}\cdots\otimes_{R}c_{n}, \nonumber \\
&&t_{n}(m\otimes_{\mathcal{K}}\widetilde{c})=m_{(0)}\otimes_{\mathcal{K}}c_{1}\otimes_R \cdots \ot_R c_n\otimes_{R}m\ns{-1}c_{0}.
\end{eqnarray}

\begin{proposition}   The modules  $_{\Kc}C^{n}(C,M)$ with cofaces, codegeneracies and cocyclic map  defined in \eqref{cocyclic-module-1} define a cocyclic module.
 \end{proposition}
\begin{proof}

Since $C$ is a left $\Kc$-module coring and comultiplication map  $\Delta$ is a left $\Kc$-module map the cofaces and the codegeneracies are well-defined.
The following computation shows that  the cyclic map  is well defined. In fact for all $k\in \Kc$ we have,
 \begin{align*}
   &t_{n}(mk\otimes_{\mathcal{K}} c_{0}\otimes_{R}\cdots\otimes_{R} c_n) \\
   &= (mk)\ns{0}\otimes_{\mathcal{K}} c_{1}\otimes_{R}\cdots\otimes_{R} c_n \ot (mk)\ns{-1}c_0\\
   &=m\ns{0}{k\ps{2}}^- \otimes_{\mathcal{K}} c_{1}\otimes_{R}\cdots\otimes_{R} c_n \ot {k\ps{2}}^+ m\ns{-1}k\ps{1}c_0\\
   &=m\ns{0}{k^-\ps{2}} \otimes_{\mathcal{K}} c_{1}\otimes_{R}\cdots\otimes_{R} c_n \ot_R k^+ m\ns{-1}k^-\ps{1}c_0\\
   &=m\ns{0} \otimes_{\mathcal{K}}{k^-\ps{2}}\triangleright( c_{1}\otimes_{R}\cdots\otimes_{R} c_n \ot_R k^+ m\ns{-1}k^-\ps{1}c_0)\\
   &=m\ns{0} \otimes_{\mathcal{K}}{k^-\ps{2}}\ps{1} c_{1}\otimes_{R}\cdots\otimes_{R} k^-\ps{2}\ps{n}c_n \ot_R k^-\ps{2}\ps{n+1}   k^+ m\ns{-1}k^-\ps{1}c_0\\
   &=m\ns{0} \otimes_{\mathcal{K}}{k^-\ps{2}} c_{1}\otimes_{R}\cdots\otimes_{R} k^-\ps{n+1}c_n \ot_R k^-\ps{n+2}   k^+ m\ns{-1}k^-\ps{1}c_0\\
   &=m\ns{0}\otimes_{\mathcal{K}} k\ps{2}c_{1}\otimes_{R}\cdots\otimes_{R} k\ps{n}c_n \ot m\ns{-1}k\ps{1}c_{0}\\
   &=t_{n}(m\otimes_{\mathcal{K}} k\ps{1}c_{0}\otimes_{R}\cdots\otimes_{R} k\ps{n}c_n) \\
   &=t_{n}(m\otimes_{\mathcal{K}} k\triangleright (c_{0}\otimes_{R}\cdots\otimes_{R} c_n)) .\\
 \end{align*}

We use the SAYD condition \eqref{SAYD-condition} in the second equality, Lemma \ref{lemma-proof}(ii) in the third equality, the diagonal action in the fifth equality. In the seventh equality we have applied the map $\Delta^{\ot n}\ot \Id$ on Lemma \ref{lemma-proof}(i).  The following computation shows the cyclicity   relation $t^{n+1}=\Id$.

\begin{align*}
  &t^{n+1}(m\ot_{\Kc}c_{0}\otimes_{R}\cdots\otimes_{R} c_n)\\
  &=m\ns{0}\ot_{\Kc} m\ns{-1} c_0 \ot_R \cdots \ot_R m\ns{-(n+1)} c_n\\
  &=m\ns{0}\ot_{\Kc} m\ns{-1}\triangleright (c_0\ot_R \cdots \ot_R c_n)\\
  &=m\ns{0}m\ns{-1}\ot_{\Kc} c_0\ot_R \cdots \ot_R c_n\\
  &= m\ot_{\Kc}c_{0}\otimes_{R}\cdots\otimes_{R} c_n.
\end{align*}
We used the stability condition in the last equality. We leave to the reader to check that $d_i$, $s_i$ and $t$ satisfy all other conditions of a cocyclic module.


\end{proof}

We denote the cyclic cohomology of the previous cocyclic module by $_{\Kc}HC^*(C,M)$.
Now we describe the cyclic cohomology of a module algebra under the symmetry of a $\times$-Hopf algebra.
A left $\Kc$-module algebra $A$ is a $\mathbb{C}$-algebra and a left $\mathcal{K}$-module satisfying the following conditions for all $k\in \Kc$, $a,a'\in A$ and $r\in R$:
\begin{enumerate}
\item[i)] $k\triangleright 1_{A}=s(\varepsilon(k))\triangleright 1_{A},$
\item[ii)] $h\triangleright (a a')= (k\ps{1}\triangleright a)(k\ps{2}\vartriangleright a'),$
\item[iii)] $(\mathfrak{t}(r)\triangleright a)a'=a(s(r)\triangleright a'),$ (multiplication is $R$-balanced).
\end{enumerate}

  For any  left $\mathcal{K}$-module algebra $A$ and a right-left SAYD module  $M$ over $\mathcal{K}$, we set
  $$_{\Kc}C^{n}(A, M)=\Hom_{R}(M\otimes_{\mathcal{K}} A^{\otimes_{R}(n+1)},R),$$ to be the set of all $R$-linear maps from $M\otimes_{\Kc} A^{\otimes_{R}(n+1)}$ to $R$. Here we consider $A^{\otimes_{R}(n+1)}$ as a left $\Kc$ module by the diagonal action.  Now we define the following cofaces, codegeneracies and cocyclic maps.
\begin{eqnarray}\label{cyclic-module algebra}
&&(\delta_{i}f)(m\otimes_{\mathcal{K}}a_{0}\otimes_{R}\cdots\otimes_{R}a_{n})= f(m\ot_{\mathcal{K}}\ot a_0\ot_R\cdots\ot_R a_i a_{i+1}\ot_R\cdots \ot_R a_n ),\nonumber \\
&&(\delta_{n}f)(m\otimes_{\Kc}a_{0}\otimes_{R}\cdots\otimes_{R}a_{n})= f(m\ns{0}\ot_{\Kc}  a_n( m\ns{-n}a_0)\ot_R\dots\ot_R m\ns{-1}a_{n-1}),  \nonumber \\
&&(\sigma_{i}f)(m\otimes_{\Kc}a_{0}\otimes_{R}\cdots\otimes_{R}a_{n})= f(m\ot_{\Kc} a_0\ot_R \cdots \ot_R 1\ot_R\cdots \ot_Ra_n),\nonumber \\
&&(\tau_{n}f)(m\otimes_{\Kc}a_{0}\otimes_{R}\cdots\otimes_{R}a_{n})= f(m\ns{0}\ot_{\Kc} a_n\ot_R  m\ns{-1}a_0\ot\dots\ot_R m\ns{-n}a_{n-1}).
\end{eqnarray}
Here we use the notation $m\ns{-i}= m\ns{-1}\ps{i}$.
\begin{proposition}
  The modules $_{\Kc}C^{n}(A, M)$ with the cofaces, codegeneracies and cocyclic maps defined in \eqref{cyclic-module algebra} define a cocyclic module.
\end{proposition}
\begin{proof}
  It is easy to see that cofaces and codegeneracies are well-defined. The following computation shows that the cocyclic map is well-defined. In fact for all $k\in \Kc$ we have,
  \begin{align*}
   &(\tau_nf)(mk\otimes_{\Kc}a_{0}\otimes_{R}\cdots\otimes_{R}a_{n})\\
&=f\big((mk)\ns{0}\ot_{\Kc} a_n\ot_R  (mk)\ns{-1}a_0\ot\dots\ot_R (mk)\ns{-n}a_{n-1}\big)\\
&=f\big(m\ns{0}k{\ps{2}}^-\ot_{\Kc} a_n \ot_R ({k\ps{2}}^+m\ns{-1}k\ps{1})\ps{1}a_0\ot_R \cdots \ot_R({k\ps{2}}^+m\ns{-1}k\ps{1})\ps{n}a_{n-1}\big)\\
&=f\big(m\ns{0}k^-{\ps{2}}\ot_{\Kc} a_n \ot_R (k^+m\ns{-1}k^-\ps{1})\ps{1}a_0\ot_R \cdots \ot_R(k^+m\ns{-1}k^-\ps{1})\ps{n}a_{n-1}\big)\\
&=f\big(m\ns{0}k^-{\ps{2}}\ot_{\Kc} a_n \ot_R k^+\ps{1}m\ns{-1}\ps{1}k^-\ps{1}\ps{1}a_0\ot_R \cdots \ot_R k^+\ps{n}m\ns{-1}\ps{n}k^-\ps{1}\ps{n}a_{n-1}\big)\\
&=f\big(m\ns{0}k^-{\ps{n+1}}\ot_{\Kc} a_n \ot_R k^+\ps{1}m\ns{-1}k^-\ps{1}a_0\ot_R \cdots \ot_R k^+\ps{n}m\ns{-n}k^-\ps{n}a_{n-1}\big)\\
&=f\big(m\ns{0}\ot_{\Kc}k^-{\ps{n+1}}\triangleright ( a_n \ot_R k^+\ps{1}m\ns{-1}k^-\ps{1}a_0\ot_R \cdots \ot_R k^+\ps{n}m\ns{-n}k^-\ps{n}a_{n-1})\big)\\
&=f\big(m\ns{0}\ot_{\Kc}k^-{\ps{n+1}}\ps{1} a_n \ot_R k^-{\ps{n+1}}\ps{2}  k^+\ps{1}m\ns{-1}k^-\ps{1}a_0\ot_R \cdots \ot_R k^-{\ps{n+1}}\ps{n+1} k^+\ps{n}m\ns{-n}k^-\ps{n}a_{n-1}\big)\\
&=f\big(m\ns{0}\ot_{\Kc}k^-{\ps{n+1}} a_n \ot_R k^-{\ps{n+2}}  k^+\ps{1}m\ns{-1}k^-\ps{1}a_0\ot_R \cdots \ot_R k^-{\ps{2n+1}} k^+\ps{n}m\ns{-n}k^-\ps{n}a_{n-1}\big)\\
&=f(m\ns{0}\ot_{\Kc} k\ps{n+1}a_n\ot_R m\ns{-1}k\ps{1}a_{0}\ot_R \cdots \ot_R m\ns{-n}k\ps{n}a_{n-1})\\
&=(\tau_nf)(m\otimes_{\Kc}k\ps{1}a_{0}\otimes_{R}\cdots\otimes_{R}k\ps{n+1}a_{n})\\
 &=(\tau_nf)(m\otimes_{\Kc}k\triangleright (a_{0}\otimes_{R}\cdots\otimes_{R}a_{n})).\\
   \end{align*}
We used the AYD condition \eqref{SAYD-condition} in the second equality, Lemma \ref{lemma-proof}(ii) in the third equality, the algebra property of the comultiplication in the forth equality, the diagonal action on seventh equality and Lemma \ref{lemma-proof}(i) in the ninth equality.
Here among all the commutativity relations between $\delta_i$, $\sigma_i$ and $\tau$, we only check the cyclicity condition. Instead of a  direct long computation for $\tau^{n+1}$, we notice that
  $\tau_{n}= \theta^{n}$, where $\theta: _{\Kc}C^{n}(A, M) \longrightarrow _{\Kc}C^{n}(A, M)$ is defined by

  \begin{equation*}
    (\theta f)(m\otimes_{\Kc}a_{0}\otimes_{R}\cdots\otimes_{R}a_{n})= f(m\ns{0}\ot_{\Kc} a_1\ot_R \cdots \ot_R a_n \ot_R m\ns{-1}a_0).
  \end{equation*}
  Furthermore,
  \begin{align*}
    &(\theta^{n+1}f)(m\otimes_{\Kc}a_{0}\otimes_{R}\cdots\otimes_{R}a_{n})\\
    &=f(m\ns{0}\ot_{\Kc}m\ns{-1}a_0\ot_R \cdots \ot_R m\ns{-1}a_n\\
    &=f(m\ns{0}\ot_{\Kc} m\ns{-1}\triangleright (a_0\ot_R \cdots \ot_R a_n))\\
    &=f(m\ns{0} m\ns{-1}\ot_{\Kc}a_0\ot_R \cdots \ot_R a_n)\\
    &= f(m\otimes_{\Kc}a_{0}\otimes_{R}\cdots\otimes_{R}a_{n}).\\
  \end{align*}
We used the stability condition in the last equality. This shows that $\theta^{n+1}=\Id$ and therefore $\tau^{n+1}= \theta^{n(n+1)} = \Id$.
\end{proof}
The cyclic cohomology of this cocyclic module is denoted by $_{\Kc}\widetilde{HC}^*(A,M)$ which generalizes the dual cyclic cohomology defined in \cite{kr1}.

\begin{example}{\rm \emph{Cyclic cohomology of a $\times_R$-Hopf algebra with coefficients in $R$}:\\
If the ground algebra $R$ has a  $\Kc$-SAYD module structure as explained in Example \ref{sayd-eg1}, then the related cocyclic module is;
\begin{equation}
  _{\Kc}C^n(\Kc,R)= R\ot_{\Kc}\underbrace{\Kc\ot_R\cdots \ot_R \Kc}_{n+1~times}\cong\underbrace{\Kc\ot_R\cdots \ot_R \Kc}_{n~times}.
\end{equation}
The isomorphism is given by
\begin{equation}
  \rho(r\ot_{\Kc}k_0\ot_R \cdots \ot_R k_n)= \mathfrak{s}(r\triangleleft k_0)k_1\ot_R\cdots \ot_R k_n,
\end{equation}
which is in fact;

\begin{equation}
   \rho(r\ot_{\Kc}k_0\ot_R \cdots \ot_R k_n)= \mathfrak{s}(\delta(\mathfrak{s}(r) k_0))k_1\ot_R\cdots \ot_R k_n,
\end{equation}
with the  inverse map which is given by
\begin{equation}
  \rho^{-1}(k_1\ot_R\cdots \ot_R k_n)= 1_R\ot_{\Kc} 1_{\Kc}\ot_{R}k_1\ot_R\cdots \ot_R k_n.
\end{equation}
Therefore the cocyclic module \eqref{cocyclic-module-1} is simplified to the following one.
\begin{eqnarray}\label{simplify-cocyclic-module-1}
&&d_{0}(k_1\ot_R\cdots \ot_R k_n)= 1_{\Kc}\ot_R k_1\otimes_{R} \otimes_{R} k_{n}, \nonumber \\
&&d_{i}(k_1\ot_R\cdots \ot_R k_n)= k_1\otimes_{R}\cdots \otimes_{R} \Delta(k_{i}) \otimes_{R}\cdots \otimes_{R} k_{n}, \nonumber \\
&&d_{n+1}(k_1\ot_R\cdots \ot_R k_n)=k_1\ot_R\cdots \ot_R k_n\ot_R \sigma,\nonumber \\
&&s_{0}(k_1\ot_R\cdots \ot_R k_n)=\mathfrak{s}(\delta(k_1))k_2\ot_R k_3\ot_R \cdots \ot_R k_n, \nonumber\\
&&s_{i}(k_1\ot_R\cdots \ot_R k_n)=k_{1}\otimes_{R}\cdots\otimes_{R}\varepsilon(k_{i})\otimes_{R}\cdots\otimes_{R}k_{n}, \nonumber \\
&&t_{n}(k_1\ot_R\cdots \ot_R k_n)=\mathfrak{s}(\delta(k_1))k_2\ot_R k_3\ot_R\cdots \ot_R k_n\ot_R \sigma. \nonumber
\end{eqnarray}

}

\end{example}
\subsection{A pairing for cyclic cohomology of $\times$-Hopf algebras}

 Let $\Kc$ be a left $\times_{R}$-Hopf algebra, $M$ be a right-left SAYD module over $\Kc$,   $A$ be a left $\Kc$-module algebra and $C$ be a left $\Kc$-module coring. Let $C$ act on $A$ from the left satisfying the following conditions for all $k\in \Kc$, $c\in C$ and  $a\in A$:
\begin{align}\label{cond1}
&(kc)a= k(ca),\\\label{cond2}
&c(ab)= (c\ps{1}a)(c\ps{2}b),\\\label{cond3}
&c\vartriangleright 1_{A}=\varepsilon(c)\vartriangleright 1_{A}.
\end{align}
 We define $B=\Hom_{\Kc}(C,A)$ to be the set of  all maps from the $R$-coring $C$ to $A$ which are both $\Kc$-linear and $R$-linear. The space $B$ is an algebra over $\mathbb{C}$ by the multiplication $\ast$ which is defined by
\begin{equation}
  (f\ast g)(c)= f(c\ps{1})f(c\ps{2}).
\end{equation}
We denote that $B$ is a $R$-bimodule by $(r\triangleright f)(c)= f(r\triangleright c)$ and similarly for the right $R$-action. There exists an unital algebra map given by $\eta_A: R\longrightarrow A,\quad \eta_A(r)=r\triangleright 1_A= s(r)\triangleright 1_A$. Therefore $B$ has the unit element $\eta_B=\eta_A \circ \ve_C$. We remind that the cyclic cohomology of the algebra $B$ is computed by the cohomology of the following cocyclic module
\begin{align*}
 &\delta^a_i(\varphi)(b_0\ot \cdots \ot b_{n+1})=\varphi(b_0\ot \cdots \ot b_i b_{i+1}\ot \cdots \ot b_{n+1}),\\
 &\sigma^a_i(\varphi)(b_0\ot \cdots \ot b_n)=\varphi(b_0\ot \cdots \ot \ve(b_i)\ot \cdots \ot b_{n}),\\
 &\tau^a_i(\varphi)(b_0\ot \cdots \ot b_n)=\varphi(b_n\ot b_0\ot \cdots \ot b_{n-1}).
\end{align*}

Let $$C^{n,n}_{a,c}=_{\Kc}C^n(A,M)\boxtimes~ _{\Kc}C^n(C,M) $$ be the diagonal complex which is a cocyclic module by $(\delta_n\ot d_n, \sigma_n\ot s_n\ot \tau_n\ot t_n)$. We define the following map;
\begin{align*}
  &\Psi_c:C^{n,n}_{a,c}\longrightarrow \Hom(B^{\ot_R(n+1) }, \mathbb{C}),\\
  &\Psi_c(\phi \boxtimes m\ot_{\Kc} c_0\ot_R \cdots \ot_R c_n)(f_0\ot_R \cdots \ot_R f_n)= \phi(m\ot_{\Kc} f_0(c_0)\ot_R \cdots \ot_R f_n(c_n)).
\end{align*}
Here $f_i\in B$, for all $ 0\leq i\leq n$ and $\phi \in _{\Kc}C^n(A,M)$. The following computation shows that the map $\Psi_c$ is well-defined:
\begin{align*}
  &\Psi_c(\phi \boxtimes mk\ot_{\Kc} c_0\ot_R \cdots \ot_R c_n)(f_0\ot_R \cdots \ot_R f_n)\\
  &= \phi(mk\ot_{\Kc} f_0(c_0)\ot_R \cdots \ot_R f_n(c_n))\\
  &=\phi(m\ot_{\Kc} k\triangleright \big(f_0(c_0)\ot_R \cdots \ot_R f_n(c_n)\big))\\
  &= \phi(m\ot_{\Kc} k\ps{1}f_0(c_0)\ot_R \cdots \ot_R k\ps{n+1}f_n(c_n))\\
  &=\phi(m\ot_{\Kc} f_0(k\ps{1}c_0)\ot_R \cdots \ot_R f_n(k\ps{n+1}c_n))\\
  &=\Psi_c(\phi \boxtimes m\ot_{\Kc}k\ps{1} c_0\ot_R \cdots \ot_R k\ps{n+1}c_n))(f_0\ot_R \cdots \ot_R f_n)\\
  &=\Psi_c(\phi \boxtimes m\ot_{\Kc}k\triangleright( c_0\ot_R \cdots \ot_R c_n))(f_0\ot_R \cdots \ot_R f_n).\\
\end{align*}
We used the diagonal action in the third equality and $\Kc$-linearity of $f_i$'s in the forth equality.
\begin{proposition}
  The map $\Psi_c$ defines a cyclic map between the diagonal cocyclic module $C^{n,n}_{a,c}$ and the  cocyclic module of the algebra $B$, \ie $C^n(B)$.
\end{proposition}
\begin{proof}
First we show that $\Psi_c$ commutes with cofaces.
\begin{align*}
  &\Psi_c(\delta_i \boxtimes d_i(\phi \ot m\ot_{\Kc} c_0 \ot_R \cdots \ot_R c_n))(f_0\ot_R\cdots \ot_R f_{n+1})\\
  &=\Psi_c(\delta_i(\phi)\boxtimes d_i(m\ot_{\Kc} c_0 \ot_R \cdots \ot_R c_n))(f_0\ot_R\cdots \ot_R f_{n+1})\\
  &=\Psi_c(\delta_i(\phi)\boxtimes m\ot_{\Kc} c_0 \ot_R \cdots \ot_R c_i\ps{1}\ot c_i\ps{2}\ot_R c_n)(f_0\ot_R\cdots \ot_R f_{n+1})\\
  &=\delta_i(\phi)(m\ot_{\Kc} f_0(c_0) \ot_R \cdots \ot_R f_i(c_i\ps{1})\ot_R f_{i+1}(c_i\ps{2})\ot_R\cdots \ot_R f_{n+1}(c_n))\\
  &=\phi(m\ot_{\Kc} f_0(c_0) \ot_R \cdots \ot_R f_i(c_i\ps{1}) f_{i+1}(c_i\ps{2})\ot_R\cdots \ot_R f_{n+1}(c_n))\\
  &=\phi(m\ot_{\Kc} f_0(c_0) \ot_R \cdots \ot_R (f_i \ast f_{i+1})(c_i) \ot_R\cdots \ot_R f_{n+1}(c_n))\\
  &=(\delta^a_i\Psi_c(\phi\boxtimes m\ot_{\Kc} c_0 \ot_R \cdots \ot_R c_n))(f_0\ot_R\cdots \ot_R f_{n+1}).\\
\end{align*}
Here we show that $\Psi_c$ commutes with codegeneracies.
\begin{align*}
  &\Psi_c(\sigma_i \boxtimes s_i(\phi \ot m\ot_{\Kc} c_0 \ot_R \cdots \ot_R c_{n}))(f_0\ot_R\cdots \ot_R f_{n-1})\\
  &=\Psi_c(\sigma_i(\phi)\boxtimes s_i(m\ot_{\Kc} c_0 \ot_R \cdots \ot_R c_n))(f_0\ot_R\cdots \ot_R f_{n-1})\\
  &=\Psi_c(\sigma_i(\phi)\boxtimes m\ot_{\Kc} c_0 \ot_R \cdots\ot_R \ve(c_i) \ot_R \cdots \ot_R c_n)(f_0\ot_R\cdots \ot_R f_n)\\
  &=\sigma_i(\phi)\left( m\ot_{\Kc} f_0(c_0) \ot_R \cdots\ot_R \ve(c_i) \ot_R \cdots \ot_R f_{n-1}(c_n)\right)\\
  &=\phi\left( m\ot_{\Kc} f_0(c_0) \ot_R \cdots\ot_R \ve(c_i)1 \ot_R \cdots \ot_R f_{n-1}(c_n)\right)\\
  &=\phi\left( m\ot_{\Kc} f_0(c_0) \ot_R \cdots\ot_R \eta_B(c_i) \ot_R \cdots \ot_R f_{n-1}(c_n)\right)\\
  &=(\sigma^a_i\Psi_c(\phi\boxtimes m\ot_{\Kc} c_0 \ot_R \cdots \ot_R c_n))(f_0\ot_R\cdots \ot_R f_{n-1}).\\
\end{align*}
The following computation shows that $\Psi_c$ commutes with the cyclic maps.
  \begin{align*}
   & \Psi_c(\tau_n \boxtimes t_n(\phi \ot m\ot_{\Kc} c_0 \ot_R \cdots \ot_R c_n))(f_0\ot_R\cdots \ot_R f_n)\\
   &=\Psi_c(\tau_n(\phi)\boxtimes t_n(m\ot_{\Kc} c_0 \ot_R \cdots \ot_R c_n))(f_0\ot_R\cdots \ot_R f_n)\\
   &=\tau\phi(m\ns{0}\ot_{\Kc} f_0(c_1)\ot_R\cdots \ot_R m\ns{-1}f_n(c_0))\\
   &=\phi(m\ns{0}\ns{0}\ot_{\Kc} m\ns{-1}f_n(c_0)\ot_R m\ns{0}\ns{-1}f_0(c_1)\ot_R\cdots \ot_R m\ns{0}\ns{-n}f_{n-1}(c_n))\\
   &=\phi(m\ns{0}\ot_{\Kc} m\ns{-1}\ps{1}f_n(c_0)\ot_R m\ns{-1}\ps{2}f_0(c_1)\ot_R\cdots \ot_R m\ns{-1}\ps{n+1}f_{n-1}(c_n))\\
   &=\phi(m\ns{0}\ot_{\Kc} m\ns{-1}\triangleright(f_n(c_0)\ot_R f_0(c_1)\cdots\ot_R f_{n-1}(c_n))\\
   &=\phi(m\ot_{\Kc} f_n(c_0)\ot_R f_0(c_1)\cdots\ot_R f_{n-1}(c_n))\\
   &=\Psi_c(\phi\boxtimes m\ot_{\Kc} c_0\ot_R \cdots \ot_R c_n)(f_n\ot_R f_0\ot_R \cdots \ot_R f_{n-1})\\
   &=(\tau^a_n\Psi_c(\phi\boxtimes m\ot_{\Kc} c_0\ot_R \cdots \ot_R c_n)(f_0\ot_R  \cdots \ot_R f_{n}).
  \end{align*}

\end{proof}

We define an unital algebra map $\lambda: A\longrightarrow B=\Hom_{\Kc}(A,C)$ given by $\lambda(a)(c)=ca$. The condition  \eqref{cond1} implies the $\Kc$-linearity of the map $\lambda(a)\in B$ and therefore $\lambda$ is well-defined. The condition \eqref{cond2} shows that the map $\lambda$ is multiplicative and finally \eqref{cond3} proves that $\lambda$ is unital. Therefore we obtain  a map of cocyclic modules $\lambda : C^*(B,\mathbb{C})\longrightarrow C^*(A, \mathbb{C})$. We set
\begin{equation}
  \Psi:= \lambda \circ \Psi_c: C_{a,c}^{n,n}\longrightarrow C^*(A, \mathbb{C}),
\end{equation}
where
$$\Psi(\phi\boxtimes m\otimes_{\Kc} c_0\ot_R\cdots \ot_R c_n )(a_0\ot\cdots \ot a_n)=\varphi(m\ot_{\Kc} c_0 a_0\ot_R \cdots \ot_R c_n a_n).$$
This map is well-defined by $\Kc$- linearity of the $C$- action of $A$ which is defined in \eqref{cond1}.

\begin{theorem} \label{main}

Let $R$ be  an unital $\mathbb{C}$-algebra, $\Kc$ be a left $\times_R$-Hopf algebra, $M$ be a right-left SAYD module over $\Kc$, $A$ be a left $\Kc$-module algebra and  $C$ be a left $\Kc$-module coring.  Let $C$ acts on $A$ satisfying \eqref{cond1}, \eqref{cond2} and \eqref{cond3}. We have the following pairing on the level of cyclic cohomology,

  $$\sqcup: \widetilde{HC}_{\Kc}^{p}(A,M)\otimes HC_{\Kc}^{q}(C,M)\longrightarrow HC^{p+q}(A),$$ given by
$$\sqcup= \Psi AW,$$ where $AW$ is the Alexander-Whitney map. There are similar pairings for Hochschild and periodic cyclic cohomology.

\end{theorem}
As an application of the pairing defined in  Theorem \ref{main}, the following corollary generalizes the Connes-Moscovici characteristic map for $\times$-Hopf algebras. This characteristic map enables us to obtain algebra cocycles by having $\times$-Hopf algebra cocycles.
\begin{corollary}
Under the conditions of Theorem \ref{main}, we have the following characteristic map on the level of cyclic cohomology.
\begin{align}\label{characteristic map}
  &~~~~~~~~~~~~~~~~~~~~~~~~~~~~HC^n_{\Kc}(\Kc, R)\longrightarrow HC^n(A)\\\nonumber
  &Tr(k_0\ot_R \cdots \ot_R k_n)(a_0\ot_R\ot \cdots \ot_R)= Tr(a_0 k_0(a_1) \cdots k_n(a_n)).
\end{align}
\end{corollary}
\begin{proof}
In  Theorem \ref{main}, let $M=R$.
As a result we have $$C^0_{\Kc}(A,R)=R\ot_R A\cong A,$$ where the isomorphism is given by $r\ot a\longmapsto \mathfrak{s}(r)\triangleright a$ and $1_R\ot a\longleftarrow a$. Therefore a $0$-Hopf cocycle is an  $\Kc$-linear map  $Tr: A\longrightarrow R$ where
\begin{equation}
  Tr(\mathfrak{s}(r)\triangleright(a_1 a_2))= Tr(a_2(\mathfrak{s}(r)\sigma \triangleright a_1)), \quad r\in R, a_1,a_2\in A.
\end{equation}
If $R$ is unital we obtain
\begin{equation}\label{sigma}
  Tr(a_1 a_2)= Tr(a_2(\sigma \triangleright a_1)), \quad  a_1,a_2\in A.
\end{equation}
Such  a trace map is called a $\sigma$-trace.  Also $Tr$ is a right $\Kc$-linear map. Since $R$ is unital we have $Tr( k\triangleright a)= Tr(a)\triangleleft k$. Using the definition of the action defined in \eqref{RSAYD} we have
\begin{equation}\label{delta}
Tr(k\triangleright a)= \delta(\mathfrak{s}(Tr(a))k).
\end{equation}
Such a trace map is called a $\delta$-trace. One notes when $\Kc$ is a Hopf algebra, this condition is equivalent to $Tr(ha)= \delta(h)Tr(a)$. In fact a $\sigma$-trace which is $\delta$-invariant is a $0$-cocycle.
Therefore  for $p=0$ in  Theorem \ref{main}   we obtain the characteristic  map defined in \eqref{characteristic map}.
\end{proof}

\section{Examples}
In this section we apply  the theoretical results in Section $2$ to some  examples of $\times$-Hopf algebras such as enveloping algebras, quantum algebraic tori, the Connes-Moscovici Hopf algebroid and the Kadison bialgebroids.
\subsection{Enveloping algebra}

Let $R$ be an algebra over the field of complex numbers. The simplest example of a left $\times_R$-Hopf algebra  which is not a Hopf algebra is $\mathcal{K}=R^e=R\ot R^{op}$  with the source and  target maps defined by
\begin{equation*}
\mathfrak{s}: R\longrightarrow \mathcal{K}, \quad r\longmapsto r\ot 1; \quad \mathfrak{t}: R^{op}\longrightarrow \mathcal{K}, \quad r\longmapsto 1\ot r,
\end{equation*}
 comultiplication defined by
 \begin{equation*}
\Delta:\mathcal{K}\longrightarrow \mathcal{K}\ot_R \mathcal{K}, \quad r_1\ot r_2\longmapsto (r_1\ot 1)\ot_R (1\ot r_2),
\end{equation*}
 counit given by
$$ \varepsilon:\mathcal{K}\longrightarrow R, \quad \varepsilon(r_1\ot r_2)=r_1r_2,  $$ and
$$\nu((r_1\ot r_2)\ot(r_3\ot r_4))= r_1\ot 1\ot r_3\ot r_4r_2,$$
$$\nu^{-1}((r_1\ot r_2)\ot(r_3\ot r_4))=r_1\ot 1\ot r_2r_3\ot r_4, $$ where $r,r_1, r_2, r_3, r_4 \in R.$\\

 \emph{ \textbf{SAYD modules over the enveloping algebras}}:
As a special case of  Example \ref{sayd-eg1} when $\Kc= R^e$,  it is shown  in \cite{HR1}  that a homogenous element $x\ot y\in R^e$ is a group-like element if and only if $xy=yx=1$. Furthermore if $\delta $ be a character then $\delta(r\ot r')= rr'$ for all $r\ot r'\in \Kc$. As a result
$^{x\ot x^{-1}}R_{\delta}$ is a right-left SAYD module over the $\times_R$-Hopf algebra $\Kc= R^e$ by the following action and coaction;
\begin{equation}
  r_2 r r_1= r \triangleleft (r_1\ot r_2), \quad ~~~~~~~r\longmapsto (rx\ot x^{-1})\ot 1,
\end{equation}
where $r,r_1,r_2, \in R$ and $x$ is an element of the center of $R$.
\\

\emph{\textbf{Cyclic cohomology of the universal algebra $R^e$}}:
   For any unital algebra $R$ over the field of complex numbers $\mathbb{C}$, we have
  \begin{equation}
   C^n_{R^e}(R^e,R) =R\ot_{R^e}\ot \underbrace{ R^e\ot_R \cdots \ot_R R^e}_{n+1~times}\cong \underbrace{R\ot \cdots \ot R}_{n~times}\ot R^{op}.
  \end{equation}
Similar to  \cite{CM01}, we use the unit map to define a map of cocyclic modules,
$$C^n(\mathbb{C}\ot \mathbb{C}^{op})\longrightarrow C^n(R\ot R^{op}).$$
   Now we  fix an unital linear functional $\varphi\in R^*$ to define a homotopy map $\mathbf{s}: C^{n}(R^e)\longrightarrow C^{n-1}(R^e)$ which is given by
  \begin{equation}
    \mathbf{s}(r_1\ot \cdots \ot r_n)= \varphi(r_1)r_2\ot \cdots \ot r_n,
  \end{equation}
  where $r_1,\cdots ,r_n\in R$. One easily checks that $\mathbf{s}$ commutes with face maps and therefore we obtain an isomorphism on the level of Hochschild and consequently  cyclic cohomology for the cocyclic modules   $C^n(\mathbb{C}\ot \mathbb{C}^{op})$ and $C^n(R\ot R^{op})$. Therefore we have,
  \begin{equation}
    HC^*(R\ot R^{op})\cong HC^*(\mathbb{C}\ot \mathbb{C}^{op})\cong HC^*(\mathbb{C}).
  \end{equation}
 This shows that the characteristic map is trivial in this case.


\subsection{Quantum algebraic torus}

  The Laurent polynomial ring in two variables,  $\mathbb{C}[U,V,U^{-1}, V^{-1}]$, is  the group ring of $\mathbb{Z}\times\mathbb{Z}$ and thus acquires a Hopf algebra structure.
   We consider a well-known deformation of this Hopf algebra which is not a Hopf algebra anymore and it is called algebraic quantum torus  denoted by $A_{\theta}$. Let us recall that $A_{\theta}$ is an unital algebra over $\mathbb{C}$ generated by two invertible elements $U$ and $V$ satisfying $UV=qVU$, where $q=e^{2\pi i\theta}$ and $\theta$ is a real number. Let $R=\mathbb{C}[U,U^{-1}]$ be the algebra of Laurent polynomials. We define $\alpha=\beta: R\longrightarrow A_{\theta}$ to be the natural embedding. One defines a coproduct $\Delta:A_{\theta}\longrightarrow A_{\theta}\ot_R A_{\theta}$ given by
  \begin{equation}
    \Delta(U^nV^m)=U^n V^m\ot_R V^m.
  \end{equation}
  The counit map $\ve: A_{\theta}\longrightarrow R$ is given by
  \begin{equation}
    \ve(U^nV^m)=U^n.
  \end{equation}
 Since the counit map is not an algebra map the quantum torus $A_{\theta}$ is not an bialgebra. Instead  it is  a left $\times_R$-bialgebroid by the coring structure defined above. Furthermore the following map
\begin{align*}
&  ~~~~~~~~~~~~~~~\nu: A_{\theta}\ot_{R^op} A_{\theta} \longrightarrow A_{\theta}\ot_R A_{\theta}\\
& U^nV^m\ot U^rV^s\longmapsto U^nV^m \ot V^m U^r V^s= q^{-mr}U^n V^m\ot U^r V^{s-m},
\end{align*}
is  bijective where the inverse map is defined by
\begin{equation}
 \nu^{-1}: U^nV^m\ot U^rV^s\longmapsto U^nV^m \ot V^{-m}U^r V^s= q^{mr} U^nV^m \ot U^rV^{s-m}
\end{equation}
 This turns $A_{\theta}$  into  a left $\times_R$-Hopf algebra.\\

 \emph{ \textbf{SAYD modules over the quantum algebraic torus}}:
  It is obvious that every element of the form $V^m$ is a group-like element in $A_{\theta}$. Furthermore, the map $: \delta: A_{\theta}\longrightarrow \mathbb{C}[U,U^{-1}]$ which is given by
  \begin{equation}
    \delta(U^nV^m)= q^{nm}U^n, \quad \quad~~~~~~~~~~~~U^nV^m\in A_{\theta},
  \end{equation}
  is a right character.
The only group-like element which satisfies the stability condition with respect to this right character is the unit element.
As a result of \eqref{RSAYD},  the following action and coaction endow $^1R_{\delta}=\mathbb{C}[U,U^{-1}]$ with  a SAYD structure on $A_{\theta}$ as follows;

\begin{equation}\label{coaction-torus}
  U^k\triangleleft U^nV^m= q^{(k+n)m}U^{k+n},\quad ~~~~~    U^n\longmapsto U^n\ot 1.
\end{equation}\\

  \emph{\textbf{Cyclic cohomology of  the quantum algebraic torus}}:
A normal Haar system $\varrho:A_{\theta}\longrightarrow \mathbb{C}[U,U^{-1}]$ was introduced for the bialgebroid structure  of the quantum algebraic torus in \cite{kr3} as follows;
\begin{equation}
  \varrho(U^nV^m)= \delta_{m,0} U^n.
\end{equation}
Therefore we have a contracting homotopy as defined in \cite{kr3} and  we obtain;
\begin{equation}
  HC^{2i+1}(A_{\theta})=0, \quad HC^{2i}(A_{\theta})=\mathbb{C}[U,U^{-1}], \quad \text{for all}\quad i\geq 0
\end{equation}
\\

In Theorem \ref{main}, let $\mathcal{K}= A_{\theta} $, $M=R=\mathbb{C}[U,U^{-1}]$ and $p=0$. One obtains the following non-trivial characteristic map;
\begin{equation}
  \mathbb{C}[U,U^{-1}]\longrightarrow HC^{2n}(A_{\theta}).
\end{equation}

\subsection{The Connes-Moscovici Hopf-algebroid}

  Here we show that the Connes-Moscovici Hopf algebroid defined in \cite{CM01} is a $\times$-Hopf algebra. Let $M$ be a smooth manifold of dimension $n$ with a finite atlas and  $FM$ be the frame bundle on $M$. Let $\Gamma_M$ denotes the pseudogroup of all local diffeomorphims $M$ where its elements are partial diffeomorphisms $\psi: \hbox{Dom}\psi\longrightarrow \hbox{Rang} \psi$, where the domain and the range of $\psi$ are both open subsets of $M$. One can lift $\psi\in \Gamma_M$ to the frame $FM$. This prolongation is denoted by  $\widetilde{\psi}$. We set
$$FM \rtimes \Gamma_M:= \{ (u,\widetilde{\varphi}), \quad \varphi\in \Gamma_M, \quad u\in    \hbox{Rang}\, \varphi  \},$$ and
$$FM ~\overline{\rtimes}~ \Gamma_M:= \{ [u,\widetilde{\varphi}], \quad \varphi\in \Gamma_M, \quad u\in \hbox{Rang}\, \varphi  \},$$ where $[u,\widetilde{\varphi}]$ stands for the class of $(u,\widetilde{\varphi})\in FM \rtimes \Gamma_M$ with respect to the following equivalence relation;
$$ (u,\widetilde{\varphi})\sim (v,\widetilde{\psi}), \quad \text{if} \quad u=v \quad \text{and}\quad \widetilde{\varphi}\mid_{W}= \widetilde{\psi}\mid_{W}.$$
Here $W$ is an open neighborhood of $u$. Let
\begin{equation}
\mathcal{A}_{FM}:= C^{\infty}_c(FM ~\overline{\rtimes}~ \Gamma_M),
\end{equation}
 be the smooth convolution algebra. Every element of this algebra is  linearly spanned by the monomials of the form $f U^*_{\psi}$ where $f\in C^{\infty}_c(Dom\widetilde{\psi})$. One has
$$f_1 U^*_{\psi_1}=f_2 U^*_{\psi_2}\quad \text{iff}\quad f_1=f_2 \quad \text{and}\quad \psi_1|_V= \psi_2|_V,$$ where $V$ is  a neighborhood of $Supp(f_1)=Supp(f_2)$. A multiplication is defined on $\mathcal{A}_{FM}$ by
\begin{equation}
f_1 U^*_{\psi_1}. f_2 U^*_{\psi_2}=f_1(f_2\circ \widetilde{\psi_1})U^*_{\psi_2 \psi_1}.
\end{equation}
We define the following algebra
\begin{equation}
  \mathcal{R}_{FM}= C^{\infty}(FM),
\end{equation}
which acts from left on $\mathcal{A}_{FM}$ by
\begin{equation}\label{leftopr}
  r\triangleright f U^*= r.f U^*, \quad r\in \mathcal{R}_{FM},
\end{equation}
and from right by
\begin{equation}\label{rightopr}
  f U^*\triangleleft r= (r\circ \widetilde{\psi}). fU^*, \quad r\in \mathcal{R}_{FM}.
\end{equation}
In fact we obtain the source map $\alpha: \mathcal{R}_{FM}\longrightarrow \mathcal{A}_{FM}$ using the left action and the target map $\beta: \mathcal{R}_{FM}^{op}\longrightarrow \mathcal{A}_{FM}$ by the right action. Also we consider the action of an arbitrary  vector field $Z$ over $FM$ which is given by
\begin{equation}\label{compopr}
  Z(fU^*)=Z(f)U^*, \quad fU^*\in \mathcal{A}_{FM}.
\end{equation}
One notes that although a vector field acts by derivations on functions on the frame bundle, the action of the vector field on $\mathcal{A}_{FM}$ is not a derivation anymore.
Now let
\begin{equation}
  \mathcal{H}_{FM}\subset \mathcal{L}(\mathcal{A}_{FM}),
\end{equation}
denotes those elements of the subalgebra of linear operators on $\mathcal{A}_{FM}$ which are generated by the three types of transformations; left multiplication, right multiplication and composition   given in \eqref{leftopr},\eqref{rightopr} and \eqref{compopr}. The elements of $\mathcal{H}_{FM}$ are called transverse differential operators on the groupoid $FM ~\overline{\rtimes}~ \Gamma_M$. One notes that $\alpha:\mathcal{R}_{FM}\longrightarrow \mathcal{H}_{FM} $ and $\beta: \mathcal{R}_{FM}^{op}\longrightarrow \mathcal{H}_{FM}$ endow $\mathcal{H}_{FM}$ with a $\mathcal{R}_{FM}$-bimodule structure. To define a $\mathcal{R}_{FM}$-coring structure on $\mathcal{H}_{FM}$, it is proved in \cite{CM01} that $\mathcal{H}_{FM}$ has a Poincare-Birkhoff-Witt-type basis over $\mathcal{R}_{FM} \ot \mathcal{R}_{FM}$ by fixing a torsion free connection on $FM$. To recall this basis, let $X_1,\cdots ,X_n$ denote  the standard vector fields corresponding to the standard basis of $\mathbb{R}^n$ and $\{Y_i^j\}$ be the fundamental vertical vector fields corresponding to the standard basis of $gl(n,\mathbb{R})$. These $n^2+n$ vectors form a basis for the tangent space of $FM$ at all points. Let $\delta_{jk}^i\in \mathcal{L}(\mathcal{A}_{FM})$ be the operators of multiplication defined in \cite{CM01}. It is proven in \cite{CM01}[Proposition 3] that transverse differential operators $\mathcal{Z}_{I}. \delta_{\kappa}$ form a basis for $\mathcal{H}_{FM}$ over $\mathcal{R}_{FM} \ot \mathcal{R}_{FM}$, where
\begin{equation}
  \mathcal{Z}_{I}=X_{i_1}\cdots X_{i_p}Y_{k_1}^{j_1} \cdots Y_{k_q}^{j_{q}}, \quad \text{and} \quad \delta_{\kappa}=\delta_{j_1k_1;\ell_1^{1}\cdots \ell_{p_1}^{1}}^{i_1}\cdots \delta_{j_rk_r;\ell_1^{r}\cdots \ell_{p_r}^{r}}^{i_r},
\end{equation}
and
\begin{equation}
  \delta_{jk;\ell_1\cdots \ell_{p_1}}^{i}= [X_{\ell_r}\cdots [X_{\ell_1},\delta_{jk}^i],\cdots].
\end{equation}
We refer the reader to \cite{CM01}[Proposition 3] for definition of multi-indices $I$ and $\kappa$. In order to define the coproduct, they have shown that the generators of $\mathcal{H}_{FM}$ acts on $\mathcal{A}_{FM}$ as a module algebra. This leads us to define a coproduct $\Delta_{FM}$ on $\mathcal{H}_{FM}$ which is not well-defined on $\mathcal{H}_{FM}\ot \mathcal{H}_{FM}$, instead the ambiguity disappears in the tensor product over $\mathcal{R}_{FM}$. In the next step, the counit is defined in \cite{CM01}[Proposition 7] by
\begin{equation}
  \ve_{FM}: \mathcal{H}_{FM}\longrightarrow \mathcal{R}_{FM}, \quad \ve(h)= h\triangleright 1.
\end{equation}
Finally the authors in \cite{CM01}[proposition 8] defined a twisted antipode $\widetilde{S}_{FM}$ by defining a faithful trace $\tau: \mathcal{A}_{FM}\longrightarrow \mathbb{C}$. As a result, $\mathcal{H}_{FM}$ is a $\times_{\mathcal{R}_{FM}}$-Lu's Hopf algebroid by
\begin{equation}
  (\Delta_{FM},\ve_{FM},\widetilde{S}_{FM} ).
\end{equation}
The following canonical map
\begin{equation}
  \mathcal{H}_{FM}\ot_{\mathcal{R}_{FM}^{op}}\mathcal{H}_{FM}\longrightarrow \mathcal{H}_{FM}\ot_{\mathcal{R}_{FM}}\mathcal{H}_{FM},
\end{equation}
which is given by
\begin{equation}
 h\ot h'\longmapsto h\ps{1}\ot h\ps{2}h'
\end{equation}
defines a $\times_{\mathcal{R}_{FM}}$-Hopf algebra structure on $\mathcal{H}_{FM}$. Here we mention a special case of the Connes-Moscovici Hopf algebroid when $M=\mathbb{R}^n$ is the flat Euclidean space. It is proved in \cite{cm-rankin} that
\begin{equation}
  \mathcal{H}_{F\mathbb{R}^n} \cong \mathcal{R}_{F\mathbb{R}^n} \rtimes \mathcal{H}_{n} \ltimes \mathcal{R}_{F\mathbb{R}^n},
\end{equation}
is a Hopf algebroid where $\Hc_n$ is the Connes-Moscovici Hopf algebra \cite{CM98}. In fact it can be shown that it is a left $\times$-Hopf algebra as follows.
Generally speaking, if $H$ is a Hopf algebra and $A$ is a $H$-module algebra then
\begin{equation}
  \mathcal{H}_{CM}= A \rtimes H \ltimes A^{op},
\end{equation}
is a left $\times_A$-Hopf algebra, called the Connes-Moscovici $\times$-Hopf algebra,  by the following structure. The algebra structure is given by
\begin{equation}
  (a\rtimes h\ltimes b)\cdot(a'\rtimes h'\ltimes b')= a(h\ps{1}a')\rtimes h\ps{2}h'\ltimes (h\ps{3}b')b
\end{equation}
The source  and target maps $\alpha: A\longrightarrow \Hc$ and $\beta:A^{op}\longrightarrow \Hc$  are given by
\begin{equation}
  \alpha(a)= a\rtimes 1\ltimes 1, \quad  \beta(a)= 1\rtimes 1\ltimes a.
\end{equation}
The coring structure is given by the following coproduct;
\begin{equation}\label{coring-CM}
  \Delta(a\rtimes h \ltimes b)= (a\rtimes h\ps{1} \ltimes 1)\ot_A ( 1\rtimes h\ps{2}\ltimes b).
\end{equation}
The counit $\ve: \Hc_{CM}\longrightarrow A$ is defined by
\begin{equation}
\ve(a\rtimes h\ltimes b)= a\ve(h)b.
\end{equation}
Furthermore we define $\nu: \Hc_{CM}\ot_{A^{op}} \Hc_{CM} \longrightarrow \Hc_{CM} \ot_A \Hc_{CM} $  by
\begin{equation}
  \nu((a\rtimes h\ltimes b) \ot_{A^{op}} (a'\rtimes h'\ltimes b'))= (a\rtimes h\ps{1} \ltimes 1) \ot_A \left(h\ps{2}a'\rtimes h\ps{3}h'\ltimes (h\ps{4}b')b\right),
\end{equation}
with the inverse map given by
\begin{equation}
  \nu^{-1}((a\rtimes h\ltimes b) \ot_{A} (a'\rtimes h'\ltimes b'))=(a\rtimes h\ps{1}\ltimes 1)\ot_{A^{op}}\left(S(h\ps{4})\triangleright (ba) \rtimes S(h\ps{4})h'\ltimes S(h\ps{2})\triangleright b'\right).
\end{equation}\\

\emph{ \textbf{SAYD modules over the Connes-Moscovici $\times$-Hopf algebra}}:
  Suppose the homogenous element $a\ot h\ot b\in A\ot H\ot A^{op}$ is a group-like element. Using the $A$-coring structure defined in \eqref{coring-CM}
  we obtain;
  \begin{equation}
   ( a\ot h\ps{1}\ot 1)\ot (1\ot h\ps{2}\ot b)= (a\ot h\ot b)\ot (a\ot h\ot b).
  \end{equation}
  This implies that $a=b=1$ and therefore $h$ is a group-like element of $H$.
        Therefore the homogenous group-like elements  of the Connes-Moscovici $\times$-Hopf algebra $H_{CM}=A\ot H\ot A^{op}$ are of the form $1\ot \sigma\ot 1$ where $\sigma$ is a group-like element of the  Hopf algebra $H$. We define  the map $\delta:H_{CM}\longrightarrow A$,  by
  \begin{equation}
    \delta(a\ot h\ot b)=\ve(h)f(ba),
  \end{equation}
where $f:A\longrightarrow A$ is an unital algebra map satisfying $f^2=f$ and $f(h\triangleright a)=\ve(h)f(a)$. One can check that $\delta$ is a right character.\\

 \emph{\textbf{Cyclic cohomology of the Connes-Moscovici  $\times$-Hopf algebras}}:
  For the Connes-Moscovici  $\times$-Hopf algebras we have;

  \begin{equation}
    HC^*(A\ot H\ot A^{op})\cong HC^*(A^e\ot H)\cong HC^*(H).
  \end{equation}
Here  the first isomorphism is the result of  Lemma \ref{lemma-isom}.  In fact as it is shown in \cite{CM01}, there is a $R$-bimodule isomorphism for the Connes-Moscovici Hopf algebra as follows;
\begin{equation}
 H_{CM}= R\ot H\ot R^{op}\cong \alpha(R)\ot \beta(R)\ot H= R\ot R^{op}\ot H.
\end{equation}
Therefore one can transfer the $\times$-Hopf algebra structure to $R\ot R^{op}\ot H$. It is easily observed that
\begin{equation}
  \Delta_{H_{CM}}= \Delta_{R^e}\ot \Delta_{H}, \quad \text{and} \quad \ve_{H_{CM}}= \ve_{R^e}\ot \ve_{H}.
\end{equation}
This shows that $H_{CM}$ is actually isomorphic to the external tensor product between the bicoalgebroid $R^e$ and the coalgebra $H$ over the complex numbers. In fact we have;
\begin{equation}
  \delta_{H_{CM}}=\delta_{R^e}\ot \delta_{H}, \quad \text{and} \quad \sigma_{H_{CM}}=\sigma_{R^e}\ot \sigma_{H}.
\end{equation}
By applying the Eilenberg-Zilber theorem we obtain
\begin{equation}
  HH^*(H_{CM})\cong HH^*(R^e)\ot HH^*(H)\cong HH^*(H),
\end{equation}
where the second isomorphism is obtained by our computations in  previous subsection. In fact the composition of the two isomorphism is given by the canonical inclusion homomorphism $f: H\longrightarrow H_{CM}$ which is also a morphism of $\times$-Hopf algebras. This enables us to obtain a map of cocyclic modules $_HC^*(H,\mathbb{C})\longrightarrow _{H_{CM}}C^*(H_{CM},R)$. Therefore the Connes-Long exact sequence of cocyclic modules relating Hochschild and cyclic cohomology implies the isomorphism on the level of cyclic cohomology.\\

   If in  Theorem \ref{main}  we set    $C=\Kc=A\ot H\ot A^{op}$, $M=A$ and $p=0$ we obtain the original  Connes-Moscovici characteristic map;
\begin{equation}
  HC^*(H)\longrightarrow HC^*(A).
\end{equation}

\subsection{The Kadison bialgebroid}
In this subsection we show that the Kadison bialgebroid $ (A\ot A^{op})\bowtie H$ introduced in \cite{kad} is a $\times_A$-Hopf algebra. Here $H$ is a Hopf algebra and $A$ is a left $H$-module algebra. First we recall the bialgebroid structure. The source map $\alpha: A\longrightarrow (A\ot A^{op})\bowtie H$ is given by
\begin{equation}
  a\longmapsto (a\ot 1)\ot 1.
\end{equation}
 The target map $\beta: A \longrightarrow (A\ot A^{op})\bowtie H$ is given by
 \begin{equation}
   a\longmapsto (1\ot a)\ot 1.
 \end{equation}
 The algebra structure is given by the following multiplication rule;
 \begin{equation}
   (a\ot b\ot h)\cdot (a'\ot b'\ot h')= a(h\ps{1}\triangleright a')\ot b'(S(h'\ps{2})\triangleright b)\ot h\ps{2}h'\ps{1}.
 \end{equation}
 The comultiplication is given by
 \begin{equation}
   (a\ot b)\ot h\longmapsto ((a\ot 1)\ot h\ps{1})\ot_A ((1\ot b)\ot h\ps{2}).
 \end{equation}
 Furthermore the counit $\ve: (A\ot A^{op})\bowtie H\longrightarrow A$
 is given by
 \begin{equation}
   (a\ot b)\ot h\longmapsto a(h\triangleright b).
 \end{equation}
 We define $\nu: A^{e}\ot H\ot_{A^{op}} A^{e}\ot H\longrightarrow A^{e}\ot H\ot_{A} A^{e}\ot H$ by
\begin{equation}
  (a\ot b\ot h)\ot_{A^{op}} (a'\ot b'\ot h')\longmapsto (a\ot 1 \ot h\ps{1})\ot_A \left((h\ps{2}\triangleright a')\ot b'(S(h'\ps{2})\triangleright b)\ot h\ps{3}h'\ps{1}\right),
\end{equation}
with the following inverse map;
\begin{equation}
  (a\ot b\ot h)\ot_{A} (a'\ot b'\ot h')\longrightarrow (a\ot 1\ot h\ps{1})\ot_{A^{op}} \left( (b\ot 1\ot S(h\ps{2})\cdot(a'\ot b'\ot h')\right)
\end{equation}

It is proved by Panaite and Van Oystaeyen in \cite{pvo} that the Connes-Moscovici bialgebroid  is isomorphic to the Kadison bialgebroid;
\begin{equation}
  A \rtimes H \ltimes A^{op}\cong (A\ot A^{op})\bowtie H.
\end{equation}
This isomorphism is given by
\begin{equation}\label{isom1}
  \chi: (A\ot A^{op})\bowtie H\longrightarrow A\ot H\ot A^{op}, \quad a\ot b\ot h\longmapsto  a\ot h\ps{1}\ot h\ps{2}b,
\end{equation}
and
\begin{equation}
  \chi^{-1}: a\ot h\ot b\longmapsto a\ot S(h\ps{2})\triangleright b\ot h\ps{1}.
\end{equation}
It is mentioned in \cite{pvo} that if $S^2=\Id$ then the map $\chi$ induces an isomorphism  on the level of B\"ohm-Szlach\'anyi Hopf algebroid.
We recall that two $\times$-Hopf algebras $\Kc$ and $\Hc$  are isomorphic if there exists a map $ \zeta: \Kc\longrightarrow \Hc$ which commutes with all bialgebroid structures and furthermore $\zeta \nu= \nu \zeta$. Similarly one has the following statement.

\begin{proposition}\label{lemma-isom}
 The Connes-Moscovici and the Kadison $\times$-Hopf algebras $A \rtimes H \ltimes A^{op}$ and $ (A\ot A^{op})\bowtie H$ are isomorphic.
\end{proposition}

Since the Connes-Moscovici and Kadison $\times$-Hopf algebras are isomorphic their cyclic cohomology are also  isomorphic.


\end{document}